\documentclass[11pt,leqno,oneside,letterpaper]{amsart}
\usepackage[letterpaper,left=1.75truein, top=1truein]{geometry}
\usepackage{amsmath,amsfonts,amssymb,amscd,amsthm,amsbsy,epsf,graphics}
\usepackage{xy}
\xyoption{all}

\usepackage{color}

\usepackage[colorlinks,linkcolor=blue,citecolor=blue]{hyperref}
\usepackage{epsfig}
\usepackage{graphicx}
\usepackage{ifpdf}
\usepackage{epstopdf}

\newtheorem{thm}{Theorem}[section]

\newtheorem{lemma}[thm]{Lemma}

\newtheorem{que}[thm]{Question}
\newtheorem{rem}[thm]{Remark}

\theoremstyle{definition}

\def\qed{{\hspace{2mm}{\small $\diamondsuit$}}}

\newtheoremstyle{cases}
  {12pt plus 6 pt}
  {2pt}
  {\bfseries}   
  {}
  {\bfseries}
  {.}
  {.5em}
  {}
\theoremstyle{cases}

\newtheorem{subcase}{Subcase}

\numberwithin{subcase}{section}
\numberwithin{equation}{section}

\def\sfrac#1#2{\kern.1em\raise.5ex\hbox{$#1$}
    \kern-.1em/\kern-.05em\lower.25ex\hbox{$#2$}}

\def\A{{\mathcal A}}

\def\T{{\mathcal T}}
\def\wT{\widetilde{\T}}

\def\wA{\widetilde{\A}}

\def\sfrac#1#2{\kern.1em\raise.5ex\hbox{$#1$}
        \kern-.1em/\kern-.05em\lower.25ex\hbox{$#2$}}

\def\T{{\mathcal T}}

 \def\d{{\delta}}

 \def\a{{\alpha}}

 \def\ra{{\rightarrow}}

 \def\c{{\mathbb C}}
 \def\z{{\mathbb Z}}
 \def\2{{\mathbb Z_2}}

 \def\sl2{{SL(2,\mathbb C)}}
 
 \def\qed{{\hspace{2mm}{\small $\diamondsuit$}}}

 \def\pf{{\noindent{\bf Proof.\hspace{2mm}}}}
 
 \def\sm{{{\mbox{\small M}}}}
  \def\sl{{{\mbox{\small L}}}}
 \def\sc{{{\mbox{\tiny C}}}}
\def\st{{{\mbox{\tiny T}}}}
\def\sk{{{\mbox{\tiny K}}}}
\def\su{{{\mbox{\tiny U}}}}

\begin{document}

\title{The AJ-conjecture and  cabled knots over torus knots}

\author{Dennis Ruppe}
\address{Department of Mathematics, University at Buffalo, Buffalo, NY, 14214-3093, USA.}
\email{dennisru@buffalo.edu}

\author{Xingru Zhang}
\address{Department of Mathematics, University at Buffalo, Buffalo, NY, 14214-3093, USA.}
\email{xinzhang@buffalo.edu}

\maketitle

\begin{abstract}
We show that most cabled knots over torus knots in $S^3$ satisfy the AJ-conjecture,
namely each $(r,s)$-cabled knot over each  $(p,q)$-torus knot satisfies
the $AJ$-conjecture if $r$ is not a number  between
$0$ and $pqs$.  \end{abstract}

\section{Introduction}
For a knot $K$ in $S^3$, let $J_{\sk, n}(t)$ denote the \textit{$n$-colored Jones polynomial} of $K$
with the zero framing, normalized so that for the unknot $U$,
\[J_{\su,n}(t) = \frac{t^{2n}-t^{-2n}}{t^2-t^{-2}}.\]
A remarkable result, proved in \cite{GaLe}, asserts  that for every knot $K$, $J_{\sk,n}(t)$ satisfies a nontrivial linear recurrence relation.
By defining $J_{\sk,-n}(t) := -J_{\sk,n}(t)$, one may  treat $J_{\sk,n}(t)$ as
 a discrete function  $$J_{\sk,-}(t): \z \to \z[t^{\pm 1}].$$
 The quantum torus
\[ \T = \c[t^{\pm 1}] \left<\sl^{\pm 1}, \sm^{\pm 1} \right> / (\sl \sm-t^2\sm \sl)\]
acts on   the set of discrete functions $f: \z \to \c[t^{\pm 1}]$ by
\[(\sm f)(n) := t^{2n}f(n), \quad  (\sl f)(n) := f(n+1).\]
 Then   linear recurrence relations of $J_{\sk,n}(t)$
 correspond naturally to
  annihilators of $J_{\sk,n}(t)$ in $\T$.
The latter set, which we denote by
  $${\mathcal A}_\sk := \{P \in \T \mid P J_{\sk,n}(t) = 0\},$$
  is obviously  a left ideal of $\T$, called the \textit{recurrence ideal} of $K$.
  The result of \cite{GaLe} cited above states that ${\mathcal A}_\sk$
  is not the zero ideal for every knot $K$.

The ring $\T$ can be extended to a principal left ideal domain $\wT$ by adding inverses of polynomials in $t$ and $\sm$; that is, $\wT$ is the set of Laurent polynomials in $\sl$ with coefficients rational functions of $t$ and $\sm$ with a product defined by
\[f(t,\sm) \sl^a \cdot g(t,\sm) \sl^b = f(t,\sm) g(t,t^{2a}\sm)\sl^{a+b}.\]
The left ideal $\wA_\sk=\wT {\mathcal A_\sk}$ is then generated by some nonzero polynomial in $\wT$, and in particular, this generator can be chosen to be in ${\mathcal A}_\sk$ and be of the form
\[\alpha_\sk(t,\sm,\sl) = \sum_{i=0}^d P_i \sl^i,\]
with $d$ minimal and with  $P_1,...,P_d \in \z[t,\sm]$ being coprime in $\z[t,\sm]$. This polynomial $\alpha_\sk$ is uniquely determined up to a sign and is called the (\textit{normalized}) \textit{recurrence polynomial} of $K$.

The $A$-polynomial was introduced in \cite{CCGLS}. For a knot $K$ in $S^3$,
its $A$-polynomial $A_\sk(\sm,\sl)\in \z[\sm,\sl]$ is a two variable polynomial
with no repeated factors and with relative prime integer coefficients,
which is uniquely associated to $K$ up to a sign. Note that $A_\sk(\sm,\sl)$ always contains the factor $\sl-1$.

The AJ-conjecture was raised  in \cite{G} which states that for every knot $K$, its recurrence polynomial $\alpha_\sk(t,\sm,\sl)$  evaluated at $t = -1$ is equal to the A-polynomial of $K$, up to a factor of a polynomial in $\sm$. The conjecture is obviously of fundamental importance as it predicts a strong connection
between two important knot invariants derived from very different backgrounds. This is also a very difficult conjecture; so far only torus knots, some $2$-bridge knots and some pretzel knots are
known to satisfy the conjecture  \cite{G} \cite {Ta} \cite{Hikami}
\cite{Le} \cite{LT} \cite{Tran}.

In this paper, we consider the AJ-conjecture for  cabled knots over torus knots. Recall that the set of nontrivial torus knots $T(p,q)$ in $S^3$
can be indexed, in a standard way,  by pairs of relative prime integers
$(p,q)$ satisfying  $|p|>q\geq 2$.
Also recall that an $(r, s)$-cabled knot on a knot $K$ in $S^3$
is the knot which can be embedded in
the boundary torus of a regular neighborhood of $K$ in $S^3$
as a curve of slope $r/s$ with respect to the meridian/longitude coordinates of $K$
satisfying $(r,s)=1$, $s\geq 2$.
Note that $r$ can be any integer relatively prime to $s$.
We have
\begin{thm}\label{main result}
The $AJ$-conjecture holds for each $(r,s)$-cabled knot $C$ over each $(p,q)$-torus knot $T$
if $r$ is not an integer between $0$ and $pqs$.
\end{thm}

A cabling formula for $A$-polynomials of cabled knots
 in $S^3$ is  given in \cite{NZ}. In particular when
 $C$ is  the $(r, s)$-cabled knot over the torus knot $T(p, q)$ in $S^3$,  its
 $A$-polynomial $A_\sc(\sm, \sl)$  is given explicitly as in (\ref{A-poly of C}) below.
For  a pair of relative prime integers $(p,q)$  with $q\geq 2$, define $F_{(p,q)}(\sm,\sl), G_{(p,q)}(\sm,\sl)\in \z[\sm,\sl]$ to be the associated
 polynomials in variables $\sm$ and $\sl$
by:
$$F_{(p, q)}(\sm, \sl):= \left\{\begin{array}
 {ll}\sm^{2p}\sl+1,\;\;&  \mbox{if $q=2$, $p>0$,}\\
\sl+\sm^{-2p},\;\;  &\mbox{if $q=2$,  $p<0$,}\\
\sm^{2pq}\sl^2-1,\;\;&  \mbox{if $q>2$, $p>0$,}\\
\sl^2-\sm^{-2pq},\;\; & \mbox{if $q>2$, $p<0$}\end{array}\right.$$and
$$G_{(p, q)}(\sm, \sl):= \left\{\begin{array}
 {ll}\sm^{pq}\sl-1,\;\;&  \mbox{if  $p>0$,}\\
\sl-\sm^{-pq},\;\; & \mbox{if  $p<0$.}\end{array}\right.
$$
Then
\begin{equation}\label{A-poly of C}A_\sc(\sm,\sl)=\left\{\begin{array}{ll}
(\sl-1)F_{(r, s)}(\sm, \sl)F_{(p,q)}(\sm^{s^2},\sl),&\mbox{if $s$ is odd};\\
(\sl-1)F_{(r, s)}(\sm, \sl)G_{(p,q)}(\sm^{s^2},\sl),&\mbox{if $s$ is even}.
\end{array}\right.
\end{equation}

A cabling formula for the $n$-colored Jones polynomial of
the $(r,s)$-cabled knot $C$ over a knot $K$ is given in \cite{Morton} (see also
\cite{Veen}) which in our normalized form is:
\begin{equation}\label{cabling formula of colored Jones}
 \begin{array}{ll}\displaystyle J_{\sc,n}(t) &=
t^{-rs(n^2-1)}\sum_{k=-\frac{n-1}{2}}^{\frac{n-1}{2}}
t^{4rk(ks+1)}J_{\sk,2ks+1}(t).\end{array}
\end{equation}
In particular the $n$-colored Jones polynomial of the $(p,q)$-torus knot $T$
 (which is the $(p,q)$-cabled knot over the unknot $U$) is:
\begin{equation}\label{colored Jones of T}
\begin{array}{ll}
J_{\st, n}(t)&=t^{-pq(n^2-1)}\sum_{k=-\frac{n-1}{2}}^{\frac{n-1}{2}}
t^{4pk(kq+1)}J_{\su,2kq+1}(t)
\\& =t^{-pq(n^2-1)}\sum_{k=-\frac{n-1}{2}}^{\frac{n-1}{2}}
t^{4pk(kq+1)}\frac{t^{4kq+2}-t^{-4qk-2}}{t^2-t^{-2}}.\end{array}\end{equation}

We divide the proof of Theorem \ref{main result} into the following cases:
\begin{enumerate}
\item $s$ is odd and $q>2$;
\item $s$ is odd and $q=2$;
\item $s>2$ is even;
\item $s=2$.
\end{enumerate}
In each case, we will find an annihilator of $J_{\sc,n}(t)$ by applying the formulas
(\ref{colored Jones of T}) and (\ref{cabling formula of colored Jones}) (where taking the general knot $K$ to be the $(p,q)$-torus knot $T$), and then proceed to prove
that it is the recurrence polynomial $\a_\sc(t,\sm,\sl)$ of $C$ when $r$ is not an integer between
$0$ and $pqs$,  making use of the degree formulas given in Section \ref{degree section}. Of course we will also compare  $\a_\sc(-1,\sm,\sl)$
with $A_\sc(\sm,\sl)$ given in (\ref{A-poly of C}) to complete the verification of
the $AJ$-conjecture for $C$.
For convenience, we often get $\alpha_\sc(t,\sm,\sl)$ in the form
$P=\sum_{i=0}^d P_iL^i\in \wA_\sc$, with $d$ minimal and with $P_i\in {\mathbb Q}(t, \sm)$
and with $P(-1,\sm,\sl)\ne 0$. Such $P$ only differs from $\a_\sc$ by a factor of a rational
 function $f(t,\sm)\in {\mathbb Q}(t, \sm)$ with $f(-1,\sm)\ne 0$
and thus  is clearly as good as the normalized
  recurrence polynomial in verification for the $AJ$-conjecture.
 We often simply call  such $P$ the recurrence polynomial of
 $C$.
 Also notice
  from the formula (\ref{A-poly of C}) that changing the sign of $r$ or $p$ only changes the $A$-polynomial of $C$ up to a power of $\sm$, so in checking
    that $P(-1,\sm,\sl)=A_\sc(\sm,\sl)$  up to a factor of a
   rational function  in $\sm$ we don't need to worry about the sign of $r$ or $p$.

Further investigation of the $AJ$-conjecture for more general  cabled knots,
such as iterated torus knots and cabled knots over some hyperbolic knots,
are being continued in \cite{R}. In particular for some cabled knots over the figure 8 knot
the $AJ$-conjecture has been verified to be true.

\section{Degrees of $J_{\st,n}(t)$ and $J_{\sc,n}(t)$}\label{degree section}

From now on in this paper,  $T$
denotes the $(p,q)$-torus knot and $C$ the $(r,s)$-cabled knot over $T$, with the index convention
given in the introduction.

For a polynomial $f(t)\in \z[t^{\pm1}]$, let $\ell[f]$ and $\hbar[f]$
 denote the lowest degree and the highest degree of $f$ in $t$ respectively.
Obviously for  $f(t),g(t)\in \z[t^{\pm1}]$,
$\ell[fg]=\ell[f]+\ell[g]$ and $\hbar[fg]=\hbar[f]+\hbar[g]$.

\begin{lemma}\label{degree of J_T} (1) When $p>q$, $$\begin{array}{l}
\ell[J_{\st,n}(t)]=-pqn^2+pq+\frac{1}{2}(1-(-1)^{n-1})(p-2)(q-2),\\
\hbar[J_{\st,n}(t)]=2(p+q-pq)|n|+2(pq-p-q).\end{array}$$
\newline
(2)
When $p<-q$, $$\begin{array}{l}
\ell[J_{\st,n}(t)]=2(p-q-pq)|n|+2(pq-p+q),\\
\hbar[J_{\st,n}(t)]=-pqn^2+pq+\frac{1}{2}(1-(-1)^{n-1})(p+2)(q-2).\end{array}$$
\end{lemma}

\pf The formula for $\ell[J_{\st,n}(t)]$ in part (1)
is  proved in \cite[Lemma 1.4]{Tran}. The rest of the lemma can be
 proved similarly.
\qed

Note that $r\ne pqs$ since $r$ is relatively prime to $s$.

\begin{lemma}\label{degree of J_C}
(1) When $p>q$, $$\begin{array}{lll}
\ell[J_{\sc,n}(t)]&=- p q s^2 n^2 + (2pqs^2-2pqs+2r-2rs)n&\\
&\;\;\;\;+2rs-2r+2pqs-pqs^2+\frac{1}{2}(1-(-1)^{(n-1)s})(p-2)(q-2),
&\mbox{if $r<pqs$},\\&&\\
\ell[J_{\sc,n}(t)]&=-rsn^2+rs+
\frac{1}{2}(1-(-1)^{(n-1)})(s-2)(r-pqs)\\
&\;\;\;\;+\frac12[1-(-1)^{(n-1)s}](p-2)(q-2), &
\mbox{if $r>pqs$},
\end{array}$$

$$\begin{array}{lll}\hbar[J_{\sc,n}(t)]&=-rsn^2+rs+\frac12(1-(-1)^{n-1})(s-2)(r-2pq+2p+2q),
&\mbox{if $r<0$}.
\end{array}$$
\newline
(2)
When $p<-q$, $$\begin{array}{lll}
\hbar[J_{\sc,n}(t)]&=- p q s^2 n^2 + (2pqs^2-2pqs+2r-2rs)n&\\
&\;\;\;\;+2rs-2r+2pqs-pqs^2+\frac{1}{2}(1-(-1)^{(n-1)s})(p+2)(q-2),
&\mbox{if $r>pqs$},\\&&\\
\hbar[J_{\sc,n}(t)]&=-rsn^2+rs+
\frac{1}{2}(1-(-1)^{(n-1)})(s-2)(r-pqs)\\
&\;\;\;\;+\frac12(1-(-1)^{(n-1)s})(p+2)(q-2), &
\mbox{if $r<pqs$},
\end{array}$$
$$\begin{array}{lll}\ell[J_{\sc,n}(t)]&=-rsn^2+rs+\frac12(1-(-1)^{n-1})(s-2)(r-2pq+2p-2q),
&\mbox{if $r>0$}.
\end{array}$$
\end{lemma}

\begin{proof}
(1)  From the formula (\ref{cabling formula of colored Jones}) for $J_{\sc,n}(t)$ (replacing $K$ there by $T$), we can see that
$$\ell[J_{\sc,n}(t)]= -rs(n^2-1)+\text{min}\left\{ \ell[J_{\st,2sk+1}(t)]+4rk(ks+1) \mid -\frac{n-1}{2} \leq k \leq \frac{n-1}{2}\right\}.$$
 By Lemma \ref{degree of J_T} (1), we have
\begin{align*}
\ell[J_{\st,2ks+1}(t)]+4rk(ks+1) &= -p q (2 k s+1)^2+pq+\frac{1}{2}(1-(-1)^{2 k s})(p-2) (q-2) +4 k r (k s+1)\\
&=  (4 r s-4 p q s^2)k^2+ (4 r-4 p q s)k+\frac{1}{2} (1-(-1)^{2 k s})(p-2) (q-2).
\end{align*}
When $n$ is odd, $k$ is integer valued  and thus the alternating term vanishes, so the above expression  is quadratic in $k$. When $n$ is even, $k$ is half-integer valued and the alternating term is either always equal to zero
(when $s$ is even) or is always equal to $(p-2)(q-2)$ (when $s$ is odd), and thus the above expression
  is again quadratic in $k$.
So if $r < pqs$, it is minimized at  $k = \frac{n-1}{2}$, which yields the first formula in part (1),  and  if  $r> pqs$, it is minimized at  $k = 0$
when $n$ is odd and at $k=-1/2$ when $n$ is even, which yields the second formula in part (1).

  Similarly to get the third formula in (1), we look at
$$\hbar[J_{\sc,n}(t)]= -rs(n^2-1)+\text{max}\left\{ \hbar[J_{\st,2sk+1}(t)]+4rk(ks+1) \mid -\frac{n-1}{2} \leq k \leq \frac{n-1}{2}\right\}.$$
 By Lemma \ref{degree of J_T} (1), we have
$$\begin{array}{l}
\hbar[J_{\st,2ks+1}(t)]+4rk(ks+1)
= 2(p+q-pq)|2ks+1|+2(pq-p-q) +4 k r (k s+1)\\
=\left\{\begin{array}{ll}
  4 r sk^2+ (4ps+4qs-4pqs+4r)k, &\mbox{for non-negative  $k$'s},
  \\
   4 r sk^2+ (-4ps-4qs+4pqs+4r)k+4(pq-p-q), &\mbox{for negative $k$'s.}
\end{array}\right.
\end{array}$$
If $r <0$, it is maximized at  $k = 0$
when $n$ is odd and at $k=-1/2$ when $n$ is even, which yields the third formula in part (1).

Part (2) can be proved similarly.
\end{proof}

\section{Case $s>2$ is odd and $q>2$}\label{section s odd}

\subsection{An Annihilator $P$ of $J_{\sc,n}(t)$.}
Define
\begin{align*}
\delta_j &= \frac{t^{2(p+q)(j+1)+2}+t^{-2(p+q)(j+1)+2} -t^{2(q-p)(j+1)-2}-t^{-2(q-p)(j+1)-2}}{t^2-t^{-2}},\\
S_n &= \sum_{k=1}^s t^{-4pqskn+2pqsn+4pqk^2-12pqsk+6pqs} \delta_{s(n+3)-1-2k}.
\end{align*}
By \cite[Lemma 1.1]{Tran}, we have
\begin{equation}\label{relation of J_T(n+2)}J_{\st,n+2}(t) = t^{-4pq(n+1)} J_{\st,n}(t)+t^{-2pq(n+1)} \delta_n.
\end{equation}

Note that (\ref{relation of J_T(n+2)}) is valid for every torus knot (although in \cite{Tran}, only positive
$p$ was considered).
The following two lemmas also hold for general $C$ and $T$ (without restriction on $s$ and $q$)
and they shall also be applied in later sections.

\begin{lemma} \label{J_C formulas}
\begin{align*}
J_{\sc,n+2}(t) &=
 t^{-4rs n-4rs} J_{\sc,n}(t)+ (t^{2(r-rs)n-2rs+2r-4pqs(n+1)}-t^{2(-r-rs)n-2rs-2r})J_{\st,s(n+1)-1}(t)\\
 &\;\;\;\;+t^{2(r-rs)n-2rs+2r-2pqs(n+1)} \delta_{s(n+1)-1}.
\end{align*}
\end{lemma}

\begin{proof}
We know by the cabling formula (\ref{cabling formula of colored Jones})
\begin{align*}
J_{\sc,n+2}(t) &= t^{-rs((n+2)^2-1)} \sum_{k=-\frac{n+1}{2}}^{\frac{n+1}{2}}t^{4rk(ks+1)}J_{\st,2ks+1}(t)\\
&= t^{-rs(n^2+4n+3)} \Biggl(\sum_{k=-\frac{n-1}{2}}^{\frac{n-1}{2}}t^{4rk(ks+1)}J_{\st,2ks+1}(t) + t^{4r\left(\frac{n+1}{2}\right)(\left(\frac{n+1}{2}\right)s+1)} J_{\st,s(n+1)+1}(t) \\
& \;\;\;\;+t^{4r\left(-\frac{n+1}{2}\right)(\left(-\frac{n+1}{2}\right)s+1)} J_{\st,-s(n+1)+1}(t)\Biggr).
\end{align*}
Noting that $J_{\st,-s(n+1)+1}(t) = -J_{\st,s(n+1)-1}(t)$, we have
\begin{align*}
J_{\sc,n+2}(t) &=  t^{-rs(n^2+4n+3)} \Biggl(t^{rs(n^2-1)} t^{-rs(n^2-1)}\sum_{k=-\frac{n-1}{2}}^{\frac{n-1}{2}}t^{4rk(ks+1)}J_{\st,2ks+1}(t) \\
&\;\;\;\;+ t^{(n+1)^2 rs+2r(n+1)} J_{\st,s(n+1)+1}(t)-t^{(n+1)^2 rs-2r(n+1)} J_{\st,s(n+1)-1}(t)\Biggr)\\
&= t^{-4rs n-4rs} J_{\sc,n}(t)+ t^{2(r-rs)n-2rs+2r}J_{\st,s(n+1)+1}(t)-t^{2(-r-rs)n-2rs-2r}J_{\st,s(n+1)-1}(t).
\end{align*}
Since $J_{\st,s(n+1)+1}(t)$ and $J_{\st,s(n+1)-1}(t)$ are related by equation (\ref{relation of J_T(n+2)}) as
\[J_{\st,s(n+1)+1} = t^{-4pqs(n+1)} J_{\st,s(n+1)-1}+t^{-2pqs(n+1)} \delta_{s(n+1)-1},\]
we have
\begin{align*}
J_{\sc,n+2}(t) &=
 t^{-4rs n-4rs} J_{\sc,n}(t)+ (t^{2(r-rs)n-2rs+2r-4pqs(n+1)}-t^{2(-r-rs)n-2rs-2r})J_{\st,s(n+1)-1}(t)\\
 &\;\;\;\;+t^{2(r-rs)n-2rs+2r-2pqs(n+1)} \delta_{s(n+1)-1}.
\end{align*}
\end{proof}

\begin{lemma} \label{s peel identity}  For all positive integers $m$, we have
\begin{align*}
J_{\st,n}(t) &= t^{-4pqm(n+1)+4pqm(m+1)} J_{\st,n-2m}(t) + \sum_{k=1}^m t^{(-4pqk+2pq)n+4pqk^2-4pqk+2pq} \delta_{n-2k}
\end{align*}
and in particular, with $s$ any positive integer,
\[J_{\st,s(n+3)-1}(t) = t^{-4pqs^2n-8pqs^2-4pqs} J_{\st,s(n+1)-1}(t) + S_n.\]
\end{lemma}

\begin{proof} We induct on $m$. The base case $m=1$ follows directly from
 equation (\ref{relation of J_T(n+2)}). Then assume the formula holds for some positive integer $m$. Applying equation (\ref{relation of J_T(n+2)}) again yields
\begin{align*}
J_{\st,n}(t) &= t^{-4pqm(n+1)+4pqm(m+1)} J_{\st,n-2m}(t) + \sum_{k=1}^m t^{(-4pqk+2pq)n+4pqk^2-4pqk+2pq} \delta_{n-2k}\\
&= t^{-4pqm(n+1)+4pqm(m+1)} (t^{-4pq(n-2m-1)}J_{\st,n-2m-2}(t)+t^{-2pq(n-2m-1)}\delta_{n-2m-2}) \\
&\quad + \sum_{k=1}^m t^{(-4pqk+2pq)n+4pqk^2-4pqk+2pq} \delta_{n-2k}\\
&= t^{-4pqm(n+1)+4pqm(m+1)-4pq(n-2m-2+1)} J_{\st,n-2m-2}(t)\\
&\quad +t^{-4pqm(n+1)+4pqm(m+1)-2pq(n-2m-1)}\delta_{n-2m-2} + \sum_{k=1}^m t^{(-4pqk+2pq)n+4pqk^2-4pqk+2pq} \delta_{n-2k}.
\end{align*}
If we compare the terms in the summation to the $\delta_{n-2m-2}$ term outside, we can easily see that this is precisely the term where $k = m+1$. So moving it inside, we have
\begin{align*}
J_{\st,n}(t) &= t^{-4pq(m+1)(n+1)+4pq(m^2+m)+4pq(2m+2)} J_{\st,n-2(m+1)}(t) + \sum_{k=1}^{m+1} t^{(-4pqk+2pq)n+4pqk^2-4pqk+2pq} \delta_{n-2k}\\
&= t^{-4pq(m+1)(n+1)+4pq(m+1)(m+2)} J_{\st,n-2(m+1)}(t) + \sum_{k=1}^{m+1} t^{(-4pqk+2pq)n+4pqk^2-4pqk+2pq} \delta_{n-2k}
\end{align*}
as needed. Applying the formula at $s(n+3)-1$  gives the particular equation.
\end{proof}

 We shall now find an annihilator for $J_{\sc,n}(t)$.
By Lemma \ref{J_C formulas}, replacing $t^{2n}$ with $\sm$ gives us
\begin{align*}J_{\sc,n+2}(t)&=\sm^{-2rs}t^{-4rs} J_{\sc,n}(t)+(\sm^{r-rs-2pqs}t^{-2rs+2r-4pqs}-\sm^{-r-rs}t^{-2rs-2r})J_{\st,s(n+1)-1}(t)
\\&\;\;\;\;+\sm^{r-rs-pqs}t^{-2rs+2r-2pqs} \delta_{s(n+1)-1},\end{align*}
and since $J_{\sc,n+2}(t) = \sl^2 J_{\sc,n}(t)$, we find
\begin{align*}(\sl^2-\sm^{-2rs}t^{-4rs})J_{\sc,n}(t)
&=(\sm^{r-rs-2pqs}t^{-2rs+2r-4pqs}-\sm^{-r-rs}t^{-2rs-2r})J_{\st,s(n+1)-1}(t)
\\&\;\;\;\;+\sm^{r-rs-pqs}t^{-2rs+2r-2pqs} \delta_{s(n+1)-1}.\end{align*}
In this equation, let $$a(t,\sm)=\sm^{r-rs-2pqs}t^{-2rs+2r-4pqs}-\sm^{-r-rs}t^{-2rs-2r},$$
 which is the coefficient of $J_{\st,s(n+1)-1}(t)$,
then obviously $a(t,\sm)\ne 0$, and we have
\begin{equation}\label{first relation}
\begin{array}{l}a^{-1}(t,\sm)(\sl^2-\sm^{-2rs}t^{-4rs})J_{\sc,n}(t)
\\=J_{\st,s(n+1)-1}(t)
+a^{-1}(t,\sm)\sm^{r-rs-pqs}t^{-2rs+2r-2pqs}\delta_{s(n+1)-1}.\end{array}
\end{equation}
From Lemma \ref{s peel identity}, we have
$$(\sl^2-t^{-8pqs^2-4pqs}\sm^{-2pqs^2})J_{\st,s(n+1)-1}(t)=S_n.$$
So multiplying (\ref{first relation}) from the left by $(\sl^2-t^{-8pqs^2-4pqs}\sm^{-2pqs^2})$
gives
\begin{equation}\label{second relation}\begin{array}{l}
(\sl^2-t^{-8pqs^2-4pqs}\sm^{-2pqs^2})a^{-1}(t,\sm)(\sl^2-\sm^{-2rs}t^{-4rs})J_{\sc,n}(t)
\\
=S_n
+(\sl^2-t^{-8pqs^2-4pqs}\sm^{-2pqs^2})a^{-1}(t,\sm)\sm^{r-rs-pqs}t^{-2rs+2r-2pqs}\delta_{s(n+1)-1}\\
=S_n+a^{-1}(t,t^4\sm)\sm^{r-rs-pqs}t^{-6rs+6r-6pqs} \delta_{s(n+3)-1} \\
\;\;\;\;-a^{-1}(t,\sm)\sm^{r-rs-pqs-2pqs^2}t^{-2rs+2r-8pqs^2-6pqs}\delta_{s(n+1)-1}.\end{array}\end{equation}
Let $b(t,\sm)/(t^2-t^{-2})$ denote the right hand side of (\ref{second relation}).
Then $b(t,\sm)$ is a rational function in $t$ and $\sm$.
We claim that $b\ne 0$, which we show by checking $b(-1,\sm):=\lim_{t\ra -1}b(t,\sm)\ne 0$.
Recall
$$ S_n = \sum_{k=1}^s \sm^{-2pqsk+pqs}t^{4pqk^2-12pqsk+6pqs} \delta_{s(n+3)-1-2k}.$$
So we have \begin{equation}\label{S limit}\begin{array}{ll}
\lim_{t \to -1} (t^2-t^{-2}) S_n &= \sum_{k=1}^s \sm^{-2pqsk+pqs}(\sm^{s(p+q)} +\sm^{-s(p+q)}  - \sm^{s(q-p)} -\sm^{-s(q-p)}) \\
&= (\sm^{s(p+q)} +\sm^{-s(p+q)}  - \sm^{s(q-p)} -\sm^{-s(q-p)}) \frac{\sm^{pqs} (1-\sm^{-2pqs^2})}{\sm^{2pqs}-1}.\end{array}\end{equation}
Also
$$\begin{array}{l}
\lim_{t \to -1} (t^2-t^{-2})\big(a^{-1}(t,t^4\sm)\sm^{r-rs-pqs}t^{6r-6rs-6pqs} \delta_{s(n+3)-1} \\
\qquad \qquad\qquad \qquad-a^{-1}(t,\sm)\sm^{r-rs-pqs-2pqs^2}t^{-2rs+2r-8pqs^2-6pqs}\delta_{s(n+1)-1} \big)\\
= a^{-1}(-1,\sm)(\sm^{r-rs-pqs}-\sm^{r-rs-pqs-2pqs^2})(\sm^{s(p+q)} +\sm^{-s(p+q)}  - \sm^{s(q-p)} -\sm^{-s(q-p)})\\
=\frac{1}{\sm^{r-rs-2pqs}-\sm^{-r-rs}}(\sm^{r-rs-pqs}-\sm^{r-rs-pqs-2pqs^2})(\sm^{s(p+q)} +\sm^{-s(p+q)}  - \sm^{s(q-p)} -\sm^{-s(q-p)}).\end{array}
$$
Summing up the two limits above, we get
$$\begin{array}{l}
b(-1,\sm):=\lim_{t\ra -1}b(t,\sm) \\= \left( \frac{\sm^{pqs} (1-\sm^{-2pqs^2})}{\sm^{2pqs}-1}+\frac{\sm^{r-rs-pqs}-\sm^{r-rs-pqs-2pqs^2}}{\sm^{r-rs-2pqs}-\sm^{-r-rs}}\right)
(\sm^{s(p+q)} +\sm^{-s(p+q)}  - \sm^{s(q-p)} -\sm^{-s(q-p)})\\
= \frac{(-\sm^{pqs-r-rs}+
\sm^{r-rs+pqs})(1-\sm^{-2pqs^2})
(\sm^{ps}-\sm^{-ps})(\sm^{qs}-\sm^{-qs})}{(\sm^{2pqs}-1)(\sm^{r-rs-2pqs}-\sm^{-r-rs})}
,\end{array}
$$
which is not zero. So $b \neq 0$ and we conclude that our recurrence (\ref{second relation}) is   inhomogeneous. Therefore,
\begin{align*}
P(t,\sm,\sl) &= (\sl-1)b^{-1}(t,\sm)(\sl^2-t^{-8pqs^2-4pqs}\sm^{-2pqs^2})a^{-1}(t,\sm)(\sl^2-\sm^{-2rs}t^{-4rs})
\end{align*}
is an annihilator of $J_{\sc,n}(t)$ in $\wA_\sc$.

Up to this point all the results above in this section are valid for
general $C$ over $T$.
From now on in this section,  we put in the restriction that  $s$ is odd and $q > 2$.
Once we prove that $P$ is of minimal degree in $L$, it will follow that $P$ is the recurrence polynomial of $J_{\sc,n}(t)$ up to normalization. We can check the AJ-conjecture by evaluating $P$ at $t = -1$.
\[P(-1,\sm,\sl) =b^{-1}(-1,\sm)a^{-1}(-1,\sm)(\sl-1)(\sl^2-\sm^{-2pqs^2})(\sl^2-\sm^{-2rs}),\]
which, up to a nonzero rational function in $\sm$, is equal to the A-polynomial of $C$.

\subsection{$P$ is the Recurrence Polynomial of $C$.}
We now wish to show that the operator $P$ is the recurrence polynomial of $C$, up to normalization. It is enough to show that if an operator $Q = D_4 L^4 + D_3 L^3 + D_2 L^2 + D_1 L + D_0$ is an annihilator of $J_{\sc,n}(t)$
with $D_0,...,D_4\in \z[t^{\pm1},\sm^{\pm1}]$, then $Q = 0$.

Suppose $Q J_{\sc,n}(t) = 0$, that is,
\[D_4 J_{\sc,n+4}(t) + D_3 J_{\sc,n+3}(t) + D_2 J_{\sc,n+2}(t) + D_1 J_{\sc,n+1}(t) + D_0 J_{\sc,n}(t) = 0.\]
We wish to show that $D_i = 0$ for $i = 0,1,2,3,4$. Applying our Lemma \ref{J_C formulas}, we have
\begin{align*}
0 &= D_4 J_{\sc,n+4}(t) + D_3 J_{\sc,n+3}(t) + D_2 J_{\sc,n+2}(t) + D_1 J_{\sc,n+1}(t) + D_0 J_{\sc,n}(t)\\
&=D_4\big(\sm^{-2rs}t^{-12rs}J_{\sc,n+2}(t)+(\sm^{r-rs-2pqs}t^{-6rs+6r-12pqs}
-\sm^{-r-rs}t^{-6rs-6r})J_{\st,s(n+3)-1}(t)
\\&\;\;\;\;+\sm^{r-rs-pqs}t^{-6rs+6r-6pqs} \delta_{s(n+3)-1}\big)\\
&\;\;\;\;+ D_3 \big(\sm^{-2rs}t^{-8rs}J_{\sc,n+1}(t)+(\sm^{r-rs-2pqs}t^{-4rs+4r-8pqs}
-\sm^{-r-rs}t^{-4rs-4r})J_{\st,s(n+2)-1}(t)
\\&\;\;\;\;+\sm^{r-rs-pqs}t^{4r-4rs-4pqs} \delta_{s(n+2)-1}\big)\\&
\;\;\;\;+D_2J_{\sc,n+2}(t) + D_1 J_{\sc,n+1}(t) + D_0 J_{\sc,n}(t)\end{align*}
\begin{align*}&=(D_4\sm^{-2rs}t^{-12rs}+ D_2)\big (\sm^{-2rs}t^{-4rs} J_{\sc,n}(t)+(\sm^{r-rs-2pqs}t^{-2rs+2r-4pqs}-\sm^{-r-rs}t^{-2rs-2r})J_{\st,s(n+1)-1}(t)
\\&\qquad+\sm^{r-rs-pqs}t^{2r-2rs-2pqs} \delta_{s(n+1)-1}\big)\\
&\quad+D_4\big((\sm^{r-rs-2pqs}t^{-6rs+6r-12pqs}
-\sm^{-r-rs}t^{-6rs-6r})J_{\st,s(n+3)-1}(t)
+\sm^{r-rs-pqs}t^{6r-6rs-6pqs} \delta_{s(n+3)-1}\big)\\&
\quad+ D_3 \big(\sm^{-2rs}t^{-8rs}J_{\sc,n+1}(t)+(\sm^{r-rs-2pqs}t^{-4rs+4r-8pqs}
-\sm^{-r-rs}t^{-4rs-4r})J_{\st,s(n+2)-1}(t)
\\&\qquad+\sm^{r-rs-pqs}t^{4r-4rs-4pqs} \delta_{s(n+2)-1}\big)\\&
\;\;\;\; + D_1 J_{\sc,n+1}(t) + D_0 J_{\sc,n}(t)\\
&= (D_0 + D_2 \sm^{-2rs}t^{-4rs} + D_4 \sm^{-4rs}t^{-16rs})J_{\sc,n}(t) + (D_1  +  D_3\sm^{-2rs}t^{-8rs}) J_{\sc,n+1}(t)\\
&\quad+D_4(\sm^{r-rs-2pqs}t^{-6rs+6r-12pqs}
-\sm^{-r-rs}t^{-6rs-6r})J_{\st,s(n+3)-1}(t)\\
&\quad+D_3 (\sm^{r-rs-2pqs}t^{-4rs+4r-8pqs}
-\sm^{-r-rs}t^{-4rs-4r})J_{\st,s(n+2)-1}(t)\\
&\quad+(D_4\sm^{-2rs}t^{-12rs}+ D_2)(\sm^{r-rs-2pqs}t^{-2rs+2r-4pqs}-\sm^{-r-rs}t^{-2rs-2r})J_{\st,s(n+1)-1}(t)\\
&\quad+(D_4\sm^{-2rs}t^{-12rs}+ D_2)\sm^{r-rs-pqs}t^{2r-2rs-2pqs} \delta_{s(n+1)-1}+D_3\sm^{r-rs-pqs}t^{4r-4rs-4pqs}\delta_{s(n+2)-1}\\
&\qquad+D_4\sm^{r-rs-pqs}t^{6r-6rs-6pqs}\delta_{s(n+3)-1},
\end{align*}
and applying Lemma \ref{s peel identity},
\begin{align*}
&= (D_0 + D_2 \sm^{-2rs}t^{-4rs} + D_4 \sm^{-4rs}t^{-16rs})J_{\sc,n}(t) + (D_1  +  D_3\sm^{-2rs}t^{-8rs}) J_{\sc,n+1}(t)\\
&\quad+\big(D_4(\sm^{r-rs-2pqs}t^{-6rs+6r-12pqs}
-\sm^{-r-rs}t^{-6rs-6r})\sm^{-2pqs^2}t^{-8pqs^2-4pqs}\\&
\qquad+(D_4\sm^{-2rs}t^{-12rs}+ D_2)(\sm^{r-rs-2pqs}t^{-2rs+2r-4pqs}-\sm^{-r-rs}t^{-2rs-2r})\big)J_{\st,s(n+1)-1}(t)
\\
&\quad+D_3 (\sm^{r-rs-2pqs}t^{-4rs+4r-8pqs}
-\sm^{-r-rs}t^{-4rs-4r})J_{\st,s(n+2)-1}(t)\\
&\quad+(D_4\sm^{-2rs}t^{-12rs}+ D_2)\sm^{r-rs-pqs}t^{2r-2rs-2pqs} \delta_{s(n+1)-1}+D_3\sm^{r-rs-pqs}t^{4r-4rs-4pqs}\delta_{s(n+2)-1}\\
&\qquad+D_4\sm^{r-rs-pqs}t^{6r-6rs-6pqs}\delta_{s(n+3)-1}+D_4(\sm^{r-rs-2pqs}t^{-6rs+6r-12pqs}
-\sm^{-r-rs}t^{-6rs-6r})S_n
\\&= D_4' J_{\sc,n}(t) + D_3' J_{\sc,n+1}(t) + D_2' J_{\st,s(n+1)-1}(t)+D_1' J_{\st,s(n+2)-1}(t)+D_0'
\end{align*}
We claim that each $D_i' = 0$, and it then follows that each $D_i = 0$. Indeed, it follows easily from $D_3' = D_1' = 0$ that $D_3 = D_1 = 0$. For the rest, it is enough to show that the two linear equations defined by
$D_0'=0$ and $D_2'=0$ are linearly independent (with $D_2$ and $D_4$ as variables).
We can check that the determinant of the linear system is nonzero, and in particular, we multiply by $(t^2-t^{-2})$ and then evaluate at $t=-1$ in order to use equation (\ref{S limit}):
\begin{align*}
& (\sm^{r-rs-2pqs}-\sm^{-r-rs})(\sm^{r-3rs-pqs}+\sm^{r-rs-pqs}+(\sm^{r-rs-2pqs}-\sm^{-r-rs})\sm^{pqs}\frac{(1-\sm^{-2pqs^2})}{\sm^{2pqs}-1})\\
&\quad (\sm^{s(p+q)} +\sm^{-s(p+q)}  - \sm^{s(q-p)} -\sm^{-s(q-p)})\\
&-(\sm^{r-rs-2pqs-2pqs^2}-\sm^{-r-rs-2pqs^2}+\sm^{r-3rs-2pqs}-\sm^{-r-3rs})\sm^{r-rs-pqs}\\
&\quad (\sm^{s(p+q)} +\sm^{-s(p+q)}  - \sm^{s(q-p)} -\sm^{-s(q-p)})\\
=&\frac{1}{\sm^{2pqs}-1}(\sm^{s(p+q)} +\sm^{-s(p+q)}  - \sm^{s(q-p)} -\sm^{-s(q-p)}) (\sm^{r-2pqs}-\sm^{-r})\\
&\quad \big( \sm^{-rs}(\sm^{r-rs-pqs}(\sm^{-2rs}+1)(\sm^{2pqs}-1)+(\sm^{r-rs-pqs}-\sm^{-r-rs+pqs})(1-\sm^{-2pqs^2}))\\
&\quad -\sm^{r-rs-pqs}(\sm^{-2pqs^2-rs}+\sm^{-3rs})(\sm^{2pqs}-1)\big),
\end{align*}
expanding some of the terms to observe cancelation,
\begin{align*}
=&\frac{1}{\sm^{2pqs}-1}(\sm^{s(p+q)} +\sm^{-s(p+q)}  - \sm^{s(q-p)} -\sm^{-s(q-p)}) (\sm^{r-2pqs}-\sm^{-r})\\
&\quad \big( \sm^{r-4rs+pqs} - \sm^{r-4rs-pqs}+\sm^{r-2rs+pqs}-\sm^{r-2rs-pqs} \\
&\qquad +\sm^{r-2rs-pqs}-\sm^{r-2rs-pqs-2pqs^2} -\sm^{-r-2rs+pqs} + \sm^{-r-2rs+pqs-2pqs^2} \\
&\qquad +\sm^{r-2rs-pqs-2pqs^2} + \sm^{r-4rs-pqs} - \sm^{r-2rs+pqs-2pqs^2}-\sm^{r-4rs+pqs} \big)\\
=&\frac{(\sm^{r-2pqs}-\sm^{-r})}{\sm^{2pqs}-1}
(\sm^{ps}-\sm^{-ps})(\sm^{qs}-\sm^{-qs})
\big(  (\sm^{-2pqs^2}-1)(\sm^{-r-2rs+pqs} -\sm^{r-2rs+pqs}) \big)
\end{align*}
which is indeed nonzero.

We now prove that if $r$ is not an integer between $0$ and $pqs$, we have $D_i' = 0$ for each $i = 0,1,2,3,4$ and thus our annihilator $P$ is of minimal $\sl$-degree.

We say that a function $f:\z \to \z$ is a \textit{quasi-polynomial} if there exist periodic functions $a_0, \ldots, a_d$ each with integral period such that
\[f(n) = \sum_{i=0}^d a_i(n) n^i,\]
and $f$ is of degree $d$ if $a_d \neq 0$. In particular, we say $f$ is \textit{quasi-quadratic} if $f$ is a quasi-polynomial of degree 2.

\begin{lemma} \label{degree argument for s odd}(1) When $p > q$ and either $r < 0$ or $r > pqs$, we have $D_i' = 0$ for $i = 0,1,2,3,4$.

(2) When $p < -q$ and either $r > 0$ or $r < pqs$, we have $D_i' = 0$ for $i = 0,1,2,3,4$.
\end{lemma}

\begin{proof} (1) Suppose $p > q$, $r > pqs$, and some $D_i' \neq 0$. Then
there must be another nonzero $D_j'$ such that one of the following equalities hold:
\begin{align*}
\ell[D_4' J_{\sc,n}(t)] &= \ell[D_3' J_{\sc,n+1}(t)],\\
\ell[D_4' J_{\sc,n}(t)] &= \ell[D_2' J_{\st,s(n+1)-1}(t)],\\
\ell[D_4' J_{\sc,n}(t)] &= \ell[D_1' J_{\st,s(n+2)-1}(t)],\\
\ell[D_3' J_{\sc,n+1}(t)] &= \ell[D_2' J_{\st,s(n+1)-1}(t)],\\
\ell[D_3' J_{\sc,n+1}(t)] &= \ell[D_1' J_{\st,s(n+2)-1}(t)],\\
\ell[D_2' J_{\st,s(n+1)-1}(t)] &= \ell[D_1' J_{\st,s(n+2)-1}(t)].
\end{align*}
That is, two of the summands must share a lowest degree, and since $\ell[D_0']$ is only linear in $n$ for large enough $n$ while $\ell[J_{\sc,n}(t)]$ and $\ell[J_{\st,n}(t)]$ are quasi-quadratic
by Lemmas \ref{degree of J_C} and \ref{degree of J_T}, we can immediately dispose of the cases
 involving $D_0'$.

\begin{subcase}\label{c,n,n+1-case}$\ell[D_4' J_{\sc,n}(t)] = \ell[D_3' J_{\sc,n+1}(t)]$:\end{subcase} From the second formula of Lemma \ref{degree of J_C}(1), we have
\begin{align*}
\ell[D_4']-\ell[D_3'] &= \ell[J_{\sc,n+1}(t)] - \ell[J_{\sc,n}(t)]\\
&= -2rsn-rs-(-1)^n((s-2)(r-pqs)+(p-2)(q-2)),
\end{align*}
but for sufficiently large $n$, $\ell[D_4']-\ell[D_3']$ is a linear function in $n$, while the right hand side is not a polynomial, so we have a contradiction.

\begin{subcase}$\ell[D_4' J_{\sc,n}(t)] = \ell[D_2' J_{\st,s(n+1)-1}(t)]$: \end{subcase}From Lemmas \ref{degree of J_C}(1) and \ref{degree of J_T}(1), we have
\begin{align*}
\ell[D_4']-\ell[D_2'] &= \ell[J_{\st,s(n+1)-1}(t)] - \ell[J_{\sc,n}(t)]\\
&= s(r-p q s)n^2+(2 p q s-2 p q s^2)n - p q s^2+2p q s- r s
 \\
&\quad - \frac12(1-(-1)^{n-1})(s-2)(r-pqs),
\end{align*}
which is quasi-quadratic, while the left hand side is at most linear, giving us another contradiction.

\begin{subcase}$\ell[D_4' J_{\sc,n}(t)] = \ell[D_1' J_{\st,s(n+2)-1}(t)]$: \end{subcase}Here we have
\begin{align*}
\ell[D_4']-\ell[D_1'] &= \ell[J_{\st,s(n+2)-1}(t)]-\ell[J_{\sc,n}(t)]\\
&=   s(r - p q s)n^2 + (2 p q s - 4 p q s^2)n -p q(4s^2-4s)- r s\\
&\quad - \frac12(1-(-1)^{n-1})(s-2)(r-pqs)+ (-1)^n (p-2)(q-2)
\end{align*}
again giving us a quasi-quadratic function on the right and a linear function on the left, which is a contradiction.

\begin{subcase}$\ell[D_3' J_{\sc,n+1}(t)] = \ell[D_2' J_{\st,s(n+1)-1}(t)]$:\end{subcase} We have
\begin{align*}
\ell[D_3']-\ell[D_2'] &= \ell[J_{\st,s(n+1)-1}(t)] - \ell[J_{\sc,n+1}(t)]\\
&= s(r - p q s) n^2 + (2 p q s + 2 r s - 2 p q s^2) n -pq(s^2-2s) \\
 &\quad   - \frac12(1-(-1)^{n-1})(s-2)(r-pqs)+ (-1)^n (p-2)(q-2)
\end{align*}
which is again quasi-quadratic on the right and linear on the left,
 again a contradiction.

\begin{subcase}$\ell[D_3' J_{\sc,n+1}(t)] = \ell[D_1' J_{\st,s(n+2)-1}(t)]$:\end{subcase} We have
\begin{align*}
\ell[D_3'] - \ell[D_1'] &= \ell[J_{\st,s(n+2)-1}(t)] - \ell[J_{\sc,n+1}(t)]\\
&= (r s - p q s^2) n^2 + (2 p q s + 2 r s - 4 p q s^2) n -pq(4s^2-4s) \\
 &\quad - \frac12(1-(-1)^{n})(s-2)(r-pqs)
\end{align*}
giving us another contradiction.

\begin{subcase}\label{(t,s(n+2),s(n+1))-case}$\ell[D_2' J_{\st,s(n+1)-1}(t)] = \ell[D_1' J_{\st,s(n+2)-1}(t)]$:\end{subcase} This time we have
\begin{align*}
\ell[D_2'] - \ell[D_1'] &= \ell[J_{\st,s(n+2)-1}(t)] - \ell[J_{\st,s(n+1)-1}(t)]\\
&= -2p q s^2 n + 2 p q s - 3 p q s^2 + \frac{1}{2}(-1)^{(n+1)s}(1-(-1)^s)(p-2)(q-2)\\
&=- 2 p q s^2 n + 2 p q s - 3 p q s^2  -(-1)^n(p-2)(q-2)
\end{align*}
which is alternating on the right and eventually linear on the left, which is a contradiction.  This exhausts the possibilities of the case $r > pqs$.

Now assume $p > q$ and $r < 0$. We first consider
 the highest degrees of the summands in the equation
 $$0= D_4' J_{\sc,n}(t) + D_3' J_{\sc,n+1}(t) + D_2' J_{\st,s(n+1)-1}(t)+D_1' J_{\st,s(n+2)-1}(t)+D_0'$$
 for large positive $n$'s. If $D_3'\ne 0$ or $D_4'\ne 0$, then
 both of them cannot be zero since by Lemma \ref{degree of J_C}(1)
  $\hbar[J_{\sc, n}]$ and $\hbar[J_{\sc,n+1}]$ are each quasi-quadratic while
 by Lemma \ref{degree of J_T} (1) $\hbar[J_{\st,s(n+1)-1}(t)]$ and
 $\hbar[J_{\st,s(n+2)-1}(t)]$ are each linear in $n$ (for positive $n$'s), and
  we must have the following

 \begin{subcase}$\hbar[D_4'J_{\sc, n}(t)]=\hbar[D_3'J_{\sc,n+1}(t)]$:
 \end{subcase}
 Then we have, by Lemma \ref{degree of J_C}(1),
 $$\begin{array}{ll}
 \hbar[D_4']-\hbar[D_3']&=\hbar[J_{\sc, n+1}(t)]-\hbar[J_{\sc,n}(t)]\\
& = -2rs n -rs-(-1)^n(s-2)(r-2pq+2p+2q)
\end{array}$$
which is a linear polynomial for large $n$ on the left but is not a linear polynomial on the right, giving a contradiction.

So both $D_4'$ and $D_3'$ are zero. So we have
 $0=D_2' J_{\st,s(n+1)-1}(t) + D_1' J_{\st,s(n+2)-1}(t) + D_0'$.
We can then analyze the lowest degrees in a similar fashion as above; if one of the $D_i'$ is not zero, we must have 
\begin{subcase}$\ell[D_2' J_{\st,s(n+1)-1}(t)] = \ell[D_1' J_{\st,s(n+2)-1}(t)]$:\end{subcase}

This can be treated similarly as Subcase \ref{(t,s(n+2),s(n+1))-case}.
We conclude that each $D_i' = 0$. This completes the proof of part (1).

Part (2) of the lemma can be proved similarly with the use of Lemmas \ref{degree of J_T}(2) and \ref{degree of J_C}(2).
\end{proof}

\begin{rem}\label{subcase remark}{\rm In the proof of Lemma \ref{degree argument for s odd} we used the condition
that $s>2$  odd and $q> 2$ in several subcases. Some of these subcases will disappear
accordingly in later
sections when we impose the condition $s$ odd and $q=2$ or $s>2$ even or $s=2$.}
\end{rem}

\section{Case $s>2$ is odd and $q=2$}
\subsection{An Annihilator $P$ of $J_{\sc,n}(t)$.}
Define:
 \[U_n = \sum_{k=1}^s (-1)^{k-1} t^{2psn-4psnk+2pk^2-8psk+2pk+4ps}\frac{t^{4sn+8s-2-4k}-t^{-4sn-8s+2+4k}}{t^2-t^{-2}}.\]
 When $q = 2$, we have by \cite[Lemma 1.5]{Tran} the identity
\begin{equation} \label{q=2 J_T(n+1) relation}
J_{\st,n+1}(t) = -t^{-(4n+2)p} J_{\st,n}(t)+t^{-2pn}\frac{t^{4n+2}-t^{-4n-2}}{t^2-t^{-2}}.
\end{equation}
Note again that (\ref{q=2 J_T(n+1) relation}) is valid for negative $p$ as well.
\begin{lemma} \label{q=2 s peel lemma}When $q=2$, for all positive integers $m$, we have
\[
J_{\st,n}(t) = (-1)^m t^{(-4m n+2m^2)p}J_{\st,n-m}(t) + \sum_{k=1}^m (-1)^{k-1} t^{-(4k-2)pn+(2k^2-2k+2)p}\frac{t^{4n+2-4k}-t^{-4n-2+4k}}{t^2-t^{-2}}
\]
and in particular, when $s$ is odd,
 \[J_{\st,s(n+2)-1}(t) = - t^{-4ps^2n+4ps-6ps^2} J_{\st,s(n+1)-1}(t) + U_n.\]
\end{lemma}
\begin{proof} Apply the relation (\ref{q=2 J_T(n+1) relation}) $m$ times.
\end{proof}

Note that the relation (\ref{first relation}) is valid for general $C$ over $T$.
Specializing it at $q=2$ and $s$ odd, we have
\begin{equation}\label{first relation for q=2}
\begin{array}{l}a^{-1}(t,\sm)(\sl^2-\sm^{-2rs}t^{-4rs})J_{\sc,n}(t)
\\=J_{\st,s(n+1)-1}(t)
+a^{-1}(t,\sm)\sm^{r-rs-2ps}t^{2r-2rs-4ps}\delta_{s(n+1)-1}.\end{array}
\end{equation}
From  Lemma \ref{q=2 s peel lemma}, we get
\[(\sl+\sm^{-2ps^2}t^{4ps-6ps^2}) J_{\st,s(n+1)-1}(t)= U_n.\]
Multiplying (\ref{first relation for q=2}) from the left by
$(\sl+\sm^{-2ps^2}t^{4ps-6ps^2})$ yields
\begin{equation}\label{second relation for q=2}
\begin{array}{l}(\sl+\sm^{-2ps^2}t^{4ps-6ps^2})a^{-1}(t,\sm)(\sl^2-\sm^{-2rs}t^{-4rs})J_{\sc,n}(t)
\\=U_n+(\sl+\sm^{-2ps^2}t^{4ps-6ps^2})a^{-1}(t,\sm)\sm^{r-rs-2ps}t^{2r-2rs-4ps}\delta_{s(n+1)-1}
\\=U_n+a^{-1}(t,t^2\sm)\sm^{r-rs-2ps}t^{4r-4rs-8ps}\delta_{s(n+2)-1}\\
\quad
+a^{-1}(t,\sm)\sm^{r-rs-2ps-2ps^2}t^{2r-2rs-6ps^2}\delta_{s(n+1)-1}.\end{array}
\end{equation}

As in Section \ref{section s odd}, let $b(t,\sm)/(t^2-t^{-2})$ denote the right hand side of (\ref{second relation for q=2}).
Then $b(t,\sm)$ is a rational function in $t$ and $\sm$.
We now show that  $b\ne 0$  by checking $b(-1,\sm):=\lim_{t\ra -1}b(t,\sm)\ne 0$.
Rewrite $U_n$ as a function of $t$ and $\sm$ by changing $t^{2n}$ to $\sm$:
$$ U_n =\sum_{k=1}^s (-1)^{k-1} \sm^{ps-2psk}t^{2pk^2-8psk+2pk+4ps}\frac{\sm^{2s}t^{8s-2-4k}-\sm^{-2s}t^{-8s+2+4k}}{t^2-t^{-2}}.$$
So we have \begin{equation}\label{U limit}\begin{array}{ll}
\lim_{t \to -1} (t^2-t^{-2})U_n &=
 \sum_{k=1}^s (-1)^{k-1} \sm^{ps-2psk}
 (\sm^{2s}-\sm^{-2s})=\frac{(\sm^{2s}-\sm^{-2s})\sm^{-ps}(1+\sm^{-2ps^2})}{1+\sm^{-2ps}}.\end{array}\end{equation}
Also
$$\begin{array}{l}
\lim_{t \to -1} (t^2-t^{-2})\big(a^{-1}(t,t^2\sm)\sm^{r-rs-2ps}t^{4r-4rs-8ps}\delta_{s(n+3)-1}\\
\qquad \qquad \qquad \qquad +a^{-1}(t,\sm)\sm^{r-rs-2ps-2ps^2}t^{2r-2rs-6ps^2}\delta_{s(n+1)-1}\big)\\
=\frac{\sm^{r-rs-2ps}+\sm^{r-rs-2ps-2ps^2}}{\sm^{r-rs-4ps}-\sm^{-r-rs}}(\sm^{s(p+2)} +\sm^{-s(p+2)}  - \sm^{s(2-p)} -\sm^{-s(2-p)}).\end{array}
$$
Summing up the two limits above, we get
$$\begin{array}{l}
b(-1,\sm):=\lim_{t\ra -1}b(t,\sm) \\=\frac{(\sm^{2s}-\sm^{-2s})\sm^{-ps}(1+\sm^{-2ps^2})}{1+\sm^{-2ps}}
+\frac{\sm^{r-rs-2ps}+\sm^{r-rs-2ps-2ps^2}}{\sm^{r-rs-4ps}-\sm^{-r-rs}}(\sm^{s(p+2)} +\sm^{-s(p+2)}  - \sm^{s(2-p)} -\sm^{-s(2-p)})\\
=\frac{(\sm^{2s}-\sm^{-2s})(1+\sm^{-2ps^2})\sm^{-rs-ps}(\sm^r-\sm^{-r})}
{(1+\sm^{-2ps})(\sm^{r-rs-4ps}-\sm^{-r-rs})},\end{array}
$$
which is not zero. So $b \neq 0$ and we conclude that our recurrence (\ref{second relation for q=2}) is   inhomogeneous. Therefore,
\begin{align*}
P(t,\sm,\sl) &= (\sl-1)b^{-1}(t,\sm)(\sl+\sm^{-2ps^2}t^{4ps-6ps^2})a^{-1}(t,\sm)(\sl^2-\sm^{-2rs}t^{-4rs})
\end{align*}
is an annihilator of $J_{\sc,n}(t)$ in $\wA_\sc$.

Once we prove that $P$ is of minimal degree in $L$, it will follow that $P$ is the recurrence polynomial of $J_{\sc,n}(t)$ up to normalization. We can check the AJ-conjecture by evaluating $P$ at $t = -1$.
\[P(-1,\sm,\sl) =b^{-1}(-1,\sm)a^{-1}(-1,\sm)(\sl-1)(\sl+\sm^{-2ps^2})(\sl^2-\sm^{-2rs}),\]
which, up to a nonzero factor  in ${\mathbb Q}(\sm)$, is equal to the A-polynomial of $C$.

\subsection{$P$ is the Recurrence Polynomial of $C$.}\label{subsection s odd and q=2}
We now wish to show that the operator $P$ is the recurrence polynomial of $C$. It is enough to show that if an operator $Q = D_3 \sl^3 + D_2 \sl^2 + D_1 \sl + D_0$
 with $D_0,...,D_3\in \z[t^{\pm1},\sm^{\pm1}]$ is an annihilator of $J_{\sc,n}(t)$, then $Q = 0$.

Suppose $Q J_{\sc,n}(t) = 0$, that is,
$D_3 J_{\sc,n+3}(t) + D_2 J_{\sc,n+2}(t) + D_1 J_{\sc,n+1}(t) + D_0 J_{\sc,n}(t) = 0$.
We wish to show that $D_i = 0$ for $i = 0,1,2,3$. We have by Lemma \ref{J_C formulas}
(specialized at $q=2$)

\begin{align*}
0 &= D_3 J_{\sc,n+3}(t) + D_2 J_{\sc,n+2}(t) + D_1 J_{\sc,n+1}(t) + D_0 J_{\sc,n}(t)\\
&= D_3 \big(\sm^{-2rs}t^{-8rs}J_{\sc,n+1}(t)+(\sm^{r-rs-4ps}t^{-4rs+4r-16ps}
-\sm^{-r-rs}t^{-4rs-4r})J_{\st,s(n+2)-1}(t)
\\&\qquad+\sm^{r-rs-2ps}t^{4r-4rs-8ps} \delta_{s(n+2)-1}\big)\\&
\quad+D_2\big (\sm^{-2rs}t^{-4rs} J_{\sc,n}(t)+(\sm^{r-rs-4ps}t^{-2rs+2r-8ps}-\sm^{-r-rs}t^{-2rs-2r})J_{\st,s(n+1)-1}(t)
\\&\qquad+\sm^{r-rs-2ps}t^{2r-2rs-4ps} \delta_{s(n+1)-1}\big)\\
&\quad + D_1 J_{\sc,n+1}(t) + D_0 J_{\sc,n}(t)\\&=
 (D_0 +D_2\sm^{-2rs}t^{-4rs})J_{\sc,n}(t)+(D_1 +D_3 \sm^{-2rs}t^{-8rs})J_{\sc,n+1}(t)\\
&\quad+D_3 (\sm^{r-rs-4ps}t^{-4rs+4r-16ps}
-\sm^{-r-rs}t^{-4rs-4r})J_{\st,s(n+2)-1}(t)\\
&\quad+D_2(\sm^{r-rs-4ps}t^{-2rs+2r-8ps}-\sm^{-r-rs}t^{-2rs-2r})J_{\st,s(n+1)-1}(t)\\
&\quad+ D_2\sm^{r-rs-2ps}t^{2r-2rs-4ps} \delta_{s(n+1)-1}+D_3\sm^{r-rs-2ps}t^{4r-4rs-8ps}\delta_{s(n+2)-1},
\end{align*}
and applying Lemma \ref{q=2 J_T(n+1) relation},
\begin{align*}
&= (D_0 +D_2\sm^{-2rs}t^{-4rs})J_{\sc,n}(t)+(D_1 +D_3 \sm^{-2rs}t^{-8rs})J_{\sc,n+1}(t)\\
&\quad + D_3 (\sm^{r-rs-4ps}t^{-4rs+4r-16ps}
-\sm^{-r-rs}t^{-4rs-4r})(-\sm^{-2ps^2}t^{4ps-6ps^2}J_{\st,s(n+1)-1}(t)+U_n)\\
&\quad + D_2(\sm^{r-rs-4ps}t^{-2rs+2r-8ps}-\sm^{-r-rs}t^{-2rs-2r})J_{\st,s(n+1)-1}(t)\\
&\quad+ D_2\sm^{r-rs-2ps}t^{2r-2rs-4ps} \delta_{s(n+1)-1}+D_3\sm^{r-rs-2ps}t^{4r-4rs-8ps}\delta_{s(n+2)-1}\\
&= (D_0 +D_2\sm^{-2rs}t^{-4rs})J_{\sc,n}(t)+(D_1 +D_3 \sm^{-2rs}t^{-8rs})J_{\sc,n+1}(t)\\
&\quad +\big(D_3 (\sm^{r-rs-4ps}t^{-4rs+4r-16ps}
-\sm^{-r-rs}t^{-4rs-4r})(-\sm^{-2ps^2}t^{4ps-6ps^2})\\&\qquad+ D_2(\sm^{r-rs-4ps}t^{-2rs+2r-8ps}-\sm^{-r-rs}t^{-2rs-2r})\big)J_{\st,s(n+1)-1}(t)\\
&\quad+ D_2\sm^{r-rs-2ps}t^{2r-2rs-4ps} \delta_{s(n+1)-1}+D_3\big(
(\sm^{r-rs-4ps}t^{-4rs+4r-16ps}
-\sm^{-r-rs}t^{-4rs-4r})U_n\\&\qquad+\sm^{r-rs-2ps}t^{4r-4rs-8ps}\delta_{s(n+2)-1}\big)\\
&=D_3'J_{\sc,n}(t)+D_2'J_{\sc,n+1}(t)+D_1'J_{\st,s(n+1)-1}(t)+D_0'
\end{align*}
We claim that each $D_i' = 0$, and it then follows as in the previous section that each $D_i = 0$.
We again wish to show that the two linear equations defined by
$D_0'=0$ and $D_1'=0$ are linearly independent (with $D_2$ and $D_3$ as variables).
So let's check the determinant of the linear system, multiplied by $(t^2-t^{-2})$
and then valued at $t=-1$, is nonzero:
\begin{align*}
&(\sm^{r-rs-4ps}
-\sm^{-r-rs})(-\sm^{-2ps^2})
\sm^{r-rs-2ps}(\sm^{s(p+2)} +\sm^{-s(p+2)}  - \sm^{s(2-p)} -\sm^{-s(2-p)})\\
&-(\sm^{r-rs-4ps}-\sm^{-r-rs})\big((\sm^{r-rs-4ps}-\sm^{-r-rs})
\frac{(\sm^{2s}-\sm^{-2s})\sm^{-ps}(1+\sm^{-2ps^2})}{1+\sm^{-2ps}}\\
&\quad+\sm^{r-rs-2ps}(\sm^{s(p+2)} +\sm^{-s(p+2)}  - \sm^{s(2-p)} -\sm^{-s(2-p)})\big)\\
=&-(\sm^{r-rs-4ps}
-\sm^{-r-rs})(\sm^{-2ps^2}+1)\big(\sm^{r-rs-2ps}
(\sm^{s(p+2)} +\sm^{-s(p+2)}  - \sm^{s(2-p)} -\sm^{-s(2-p)})\\
&+(\sm^{r-rs-4ps}-\sm^{-r-rs})
\frac{(\sm^{2s}-\sm^{-2s})\sm^{-ps}}{1+\sm^{-2ps}}\big)
\end{align*}
\begin{align*}
=&-(\sm^{r-rs-4ps}-\sm^{-r-rs})(\sm^{-2ps^2}+1)(\sm^{2s}-\sm^{-2s})\big(\sm^{r-rs-2ps}
(\sm^{ps}-\sm^{-ps})\\
&+(\sm^{r-rs-4ps}-\sm^{-r-rs})
\frac{1}{\sm^{ps}+\sm^{-ps}}\big)\\
=&-\frac{1}{\sm^{ps}+\sm^{-ps}}(\sm^{r-rs-4ps}-\sm^{-r-rs})(\sm^{-2ps^2}+1)(\sm^{2s}-\sm^{-2s})\big(\sm^{r-rs-2ps}
(\sm^{2ps}-\sm^{-2ps})\\
&+(\sm^{r-rs-4ps}-\sm^{-r-rs})\big)\\
=&-\frac{1}{\sm^{ps}+\sm^{-ps}}(\sm^{r-rs-4ps}-\sm^{-r-rs})(\sm^{-2ps^2}+1)(\sm^{2s}-\sm^{-2s})(\sm^{r-rs}-\sm^{-r-rs})
\end{align*}
which is indeed nonzero.

The following lemma shows that each $D_i' = 0$ if $r$ is not a number between $0$ and $pqs$.

\begin{lemma} \label{q=2 min degree lemma} (1) When $p > q = 2$ and either $r < 0$ or $r > pqs$, we have $D_i' = 0$ for $i = 0,1,2,3$.

(2) When $p < -q$ and either $r > 0$ or $r < pqs$, we have $D_i' = 0$ for $i = 0,1,2,3$.
\end{lemma}

\begin{proof} The proof is entirely similar to that of Lemma \ref{degree argument for s odd}
(also cf Remark \ref{subcase remark}).
\end{proof}

\section{Case $s>2$ is even}
\subsection{An Annihilator $P$ of $J_{\sc,n}(t)$.}
Recall our definition
\[\delta_j = \frac{t^{2(p+q)(j+1)+2}+t^{-2(p+q)(j+1)+2} -t^{2(q-p)(j+1)-2}-t^{-2(q-p)(j+1)-2}}{t^2-t^{-2}}\]
and further define:
\[V_n = \sum_{k=1}^{s/2}t^{-4pqsnk+2pqsn+4pqk^2-8pqsk+4pqs} \delta_{s(n+2)-1-2k}.\]

\begin{lemma} \label{s/2 peel identity} If $s$ is even, then
\[J_{\st,s(n+2)-1}(t) = t^{-2pqs^2n-3pqs^2+2pqs} J_{\st,s(n+1)-1}(t) + V_n \]
 \end{lemma}
\begin{proof} Apply Lemma \ref{s peel identity}, setting $m = s/2$. \end{proof}

 Lemma \ref{s/2 peel identity} yields  the relation
 $$(\sl-\sm^{-pqs^2}t^{-3pqs^2+2pqs}) J_{\st,s(n+1)-1}(t)=V_n.$$
 So applying the operator $(\sl-\sm^{-pqs^2}t^{-3pqs^2+2pqs})$
  to both sides of (\ref{first relation}) gives:
$$\begin{array}{l}(\sl-\sm^{-pqs^2}t^{-3pqs^2+2pqs})
a^{-1}(t,\sm)(\sl^2-\sm^{-2rs}t^{-4rs})J_{\sc,n}(t)
\\=V_n+(\sl-\sm^{-pqs^2}t^{-3pqs^2+2pqs})a^{-1}(t,\sm)\sm^{r-rs-pqs}t^{2r-2rs-2pqs}\delta_{s(n+1)-1}\\
=V_n+a^{-1}(t,t^2\sm)\sm^{r-rs-pqs}t^{4r-4rs-4pqs}\delta_{s(n+2)-1}
-a^{-1}(t,\sm)\sm^{r-rs-pqs^2-pqs}t^{2r-2rs-3pqs^2}\delta_{s(n+1)-1}.\end{array}$$
To see this is an inhomogeneous recursion for $J_{\sc,n}(t)$, let $b(t,\sm)/(t^2-t^{-2})$ be the right hand side of this equation and check it is nonzero.
As before it suffices to check that $\lim_{t\ra -1}b(t,\sm)\ne0$,
considering $V_n$ and $\d_{s(n+k)-j}$'s as functions of $t$ and $\sm$ (changing $t^{2n}$ to $\sm$).
We have $$\begin{array}{l}b(-1,\sm)=\lim_{t\ra-1}b(t,\sm)\\
=\lim_{t\ra -1}(t^2-t^{-2})(V_n+a^{-1}(t,t^2\sm)\sm^{r-rs-pqs}t^{4r-4rs-4pqs}\delta_{s(n+2)-1}
\\ \qquad -a^{-1}(t,\sm)\sm^{r-rs-pqs^2-pqs}t^{2r-2rs-3pqs^2}\delta_{s(n+1)-1})\\
=(\sum_{i=1}^{s/2}\sm^{-2pqsj+pqs}+a^{-1}(-1,\sm)(\sm^{r-rs-pqs}-\sm^{r-rs-pqs^2-pqs}))\\
\qquad \times (\sm^{s(p+q)} +\sm^{-s(p+q)}  - \sm^{s(q-p)} -\sm^{-s(q-p)})\\
=\left(\frac{\sm^{-pqs}-\sm^{-pqs^2-pqs}}
{1-\sm^{-2pqs}}+\frac{\sm^{r-rs-pqs}-\sm^{r-rs-pqs^2-pqs}}{
\sm^{r-rs-2pqs}-\sm^{-r-rs}}\right)
(\sm^{s(p+q)} +\sm^{-s(p+q)}  - \sm^{s(q-p)} -\sm^{-s(q-p)})\\
=\frac{(1-\sm^{-pqs^2})\sm^{-rs-pqs}(\sm^r-\sm^{-r})(\sm^{ps}-\sm^{-ps})(\sm^{qs}-\sm^{-qs})}
{(1-\sm^{-2pqs})(\sm^{r-rs-2pqs}-\sm^{-r-rs})},\end{array}$$
which is indeed nonzero.
Hence
\begin{align*}
P(t,\sm,\sl) &= (\sl-1)b^{-1}(t,\sm)(\sl-\sm^{-pqs^2}t^{-3pqs^2+2pqs})
a^{-1}(t,\sm)(\sl^2-\sm^{-2rs}t^{-4rs})
\end{align*}
is an annihilator of $J_{\sc,n}(t)$.
Assuming $P$ is of minimal degree in $\sl$, we can now check the AJ-conjecture by evaluating $P$ at $t = -1$. We have
\[P(-1,\sm,\sl) =b^{-1}(-1,\sm) a^{-1}(-1,\sm)(\sl-1)(\sl-\sm^{-pqs^2})(\sl^2-\sm^{-2rs}),\]
which agrees with the A-polynomial of $C$,  up to a nonzero factor of a rational function in ${\mathbb Q}(\sm)$.

\subsection{$P$ is the Recurrence Polynomial of $C$.}

We now want to show that the operator $P$ is the recurrence polynomial of $C$. It is enough to prove that if  $Q = D_3 L^3 + D_2 L^2 + D_1 L + D_0$ is an element in $\A_\sc$, then $Q = 0$. As in Subsection \ref{subsection s odd and q=2}
we have
\begin{align*}
0 &= D_3 J_{\sc,n+3}(t) + D_2 J_{\sc,n+2}(t) + D_1 J_{\sc,n+1}(t) + D_0 J_{\sc,n}(t)\\
&= (D_1  +  D_3\sm^{-2rs}t^{-8rs}) J_{\sc,n+1}(t)+(D_0 + D_2 \sm^{-2rs}t^{-4rs})J_{\sc,n}(t) \\
&\quad+D_3 (\sm^{r-rs-2pqs}t^{-4rs+4r-8pqs}
-\sm^{-r-rs}t^{-4rs-4r})J_{\st,s(n+2)-1}(t)\\
&\quad+D_2(\sm^{r-rs-2pqs}t^{-2rs+2r-4pqs}-\sm^{-r-rs}t^{-2rs-2r})J_{\st,s(n+1)-1}(t)\\
&\quad+ D_2\sm^{r-rs-pqs}t^{2r-2rs-2pqs} \delta_{s(n+1)-1}+D_3\sm^{r-rs-pqs}t^{4r-4rs-4pqs}\delta_{s(n+2)-1},
\end{align*}
and applying Lemma \ref{s/2 peel identity},
\begin{align*}
&= (D_1  +  D_3\sm^{-2rs}t^{-8rs}) J_{\sc,n+1}(t)+(D_0 + D_2 \sm^{-2rs}t^{-4rs})J_{\sc,n}(t) \\
&\quad+D_3 (\sm^{r-rs-2pqs}t^{-4rs+4r-8pqs}
-\sm^{-r-rs}t^{-4rs-4r})(\sm^{-pqs^2}t^{-3pqs^2+2pqs} J_{\st,s(n+1)-1}(t)+V_n)\\
&\quad+D_2(\sm^{r-rs-2pqs}t^{-2rs+2r-4pqs}-\sm^{-r-rs}t^{-2rs-2r})J_{\st,s(n+1)-1}(t)\\
&\quad+ D_2\sm^{r-rs-pqs}t^{2r-2rs-2pqs} \delta_{s(n+1)-1}+D_3\sm^{r-rs-pqs}t^{4r-4rs-4pqs}\delta_{s(n+2)-1}\\&
= (D_1  +  D_3\sm^{-2rs}t^{-8rs}) J_{\sc,n+1}(t)+(D_0 + D_2 \sm^{-2rs}t^{-4rs})J_{\sc,n}(t) \\
&\quad+\big(D_3 (\sm^{r-rs-2pqs}t^{-4rs+4r-8pqs}
-\sm^{-r-rs}t^{-4rs-4r})\sm^{-pqs^2}t^{-3pqs^2+2pqs}\\&\qquad
+D_2(\sm^{r-rs-2pqs}t^{-2rs+2r-4pqs}-\sm^{-r-rs}t^{-2rs-2r})\big)J_{\st,s(n+1)-1}(t)\\
&\quad+ D_2\sm^{r-rs-pqs}t^{2r-2rs-2pqs} \delta_{s(n+1)-1}+D_3(\sm^{r-rs-pqs}t^{4r-4rs-4pqs}\delta_{s(n+2)-1}\\&
\qquad+(\sm^{r-rs-2pqs}t^{-4rs+4r-8pqs}
-\sm^{-r-rs}t^{-4rs-4r})V_n)\\&
= D_3' J_{\sc,n}(t) + D_2' J_{\sc,n+1}(t) + D_1' J_{\st,s(n+1)-1}(t) + D_0'
\end{align*}

As in the previous section, we show that  $D_i' = 0$, $i=1,...,3$,
implies $D_i = 0$, $i=0,...,3$.
We just need to show that the two linear equations  defined by
$D_0'=0$ and $D_1'=0$ are linearly independent.
Again we just need check the determinant of the linear system, multiplied by $(t^2-t^{-2})$
and then valued at $t=-1$, is nonzero:
$$\begin{array}{ll}
&\big((\sm^{r-rs-2pqs}-\sm^{-r-rs})\sm^{-pqs^2}\sm^{r-rs-pqs}
\\&-(\sm^{r-rs-2pqs}-\sm^{-r-rs})(\sm^{r-rs-pqs} +(\sm^{r-rs-2pqs}-\sm^{-r-rs})\frac{\sm^{-pqs}-\sm^{-pqs^2-pqs}}
{1-\sm^{-2pqs}})\big)
\\&(\sm^{s(p+q)} +\sm^{-s(p+q)}  - \sm^{s(q-p)} -\sm^{-s(q-p)})\\
=&\frac{(\sm^{r-rs-2pqs}-\sm^{-r-rs})}{1-\sm^{-2pqs}}\big(\sm^{r-rs-pqs^2-pqs}(1-\sm^{-2pqs})
-\sm^{r-rs-pqs}(1-\sm^{-2pqs})
\\&-(\sm^{r-rs-2pqs}-\sm^{-r-rs})(\sm^{-pqs}-\sm^{-pqs^2-pqs})\big)
(\sm^{s(p+q)} +\sm^{-s(p+q)}  - \sm^{s(q-p)} -\sm^{-s(q-p)})\\
=&(\sm^{r-rs-2pqs}-\sm^{-r-rs})(\sm^{-pqs^2}-1)(\sm^{r-rs-pqs}-\sm^{-r-rs-pqs})(\sm^{ps}-\sm^{-ps})(\sm^{qs}-\sm^{-qs})
\end{array}$$
which is indeed nonzero.

The following lemma shows that each $D_i' = 0$ if $r$ is not a number between $0$ and $pqs$.

\begin{lemma} (1) When $p > q$ and either $r < 0$ or $r > pqs$, we have $D_i' = 0$ for $i = 0,1,2,3$.

(2) When $p < -q$ and either $r > 0$ or $r < pqs$, we have $D_i' = 0$ for $i = 0,1,2,3$.
\end{lemma}

The proof is similar to that of Lemma \ref{degree argument for s odd}.

\section{Case $s=2$}
\subsection{An Annihilator $P$ of $J_{\sc,n}(t)$.}
In this section, we assume that $C$ is  a $(r,2)$-cabled knot over
a torus knot $T=T(p,q)$.
In   $$\begin{array}{ll}\displaystyle J_{\sc,n+1}(t) &=
t^{-2r((n+1)^2-1)}\sum_{k=-\frac{n}{2}}^{\frac{n}{2}}
t^{4rk(2k+1)}J_{\st,4k+1}(t),\end{array}$$
let $k=-(j+\frac12)$, then
\begin{equation}\label{s=2 first relation}\begin{array}{ll}\displaystyle J_{\sc,n+1}(t) &=
t^{-2r((n+1)^2-1)}\sum_{j=\frac{n-1}{2}}^{-\frac{n+1}{2}}
t^{4r(2j+1)j}J_{\st,-4j-1}(t)\\
 &=-
t^{-2r((n+1)^2-1)}\sum_{j=\frac{n-1}{2}}^{-\frac{n+1}{2}}
t^{4r(2j+1)j}J_{\st,4j+1}(t)\\
&=-t^{-2r((n+1)^2-1)}[t^{2rn(n+1)}J_{\st, 2n+1}+
\sum_{j=\frac{n-1}{2}}^{-\frac{n-1}{2}}
t^{4r(2j+1)j}J_{\st,4j+1}(t)]
\\
&=-t^{-2r((n+1)^2-1)}[t^{2rn(n+1)}J_{\st, -2n-1}+
t^{2r(n^2-1)}J_{\sc,n}(t)]
\\
&=t^{-2rn}J_{\st, 2n+1}-
t^{-4rn-2r}J_{\sc,n}(t).\end{array}\end{equation}
Turning $t^{2n}$ into $\sm$ and $J_{\sc,n+1}(t)$ into $\sl J_{\sc,n}(t)$, we see that
\[(\sl+\sm^{-2r}t^{-2r})J_{\sc,n}(t) = \sm^{-r} J_{\st,2n+1}(t),\]
or
\[\sm^r(\sl+\sm^{-2r}t^{-2r})J_{\sc,n}(t) = J_{\st,2n+1}(t).\]
We now wish to find an inhomogeneous recurrence for $J_{\st,2n+1}(t)$. Recall equation (\ref{relation of J_T(n+2)}):
\[J_{\st,n+2}(t) = t^{-4pq(n+1)} J_{\st,n}(t)+t^{-2pq(n+1)} \delta_n,\]
which implies that
\begin{equation}\label{s=2 second relation}\begin{array}{ll}
J_{\st,2n+3}(t) &= t^{-4pq(2n+2)}J_{\st,2n+1}(t)+t^{-2pq(2n+2)} \delta_{2n+1}\\
&= \sm^{-4pq}t^{-8pq}J_{\st,2n+1}(t)+\sm^{-2pq}t^{-4pq} \delta_{2n+1},
\end{array}\end{equation}
and so
\[(\sl-\sm^{-4pq}t^{-8pq})J_{\st,2n+1}(t) = \sm^{-2pq}t^{-4pq} \delta_{2n+1}.\]
Letting $b(t,\sm)/(t^2-t^{-2}) = \sm^{-2pq}t^{-4pq} \delta_{2n+1}$.
Then   $b(t,\sm) \in \z[t^{\pm 1},\sm^{\pm 1}]$, which is obviously non-zero, and  we obtain an operator $P(t,\sm,\sl)$ which annihilates $J_{\sc,n}(t)$ given by
\[
P(t,\sm,\sl) = (\sl-1)b^{-1}(t,\sm)(\sl-\sm^{-4pq}t^{-8pq})\sm^{r}(\sl+\sm^{-2r}t^{-2r}).
\]

Assuming $P$ has the minimal $\sl$ degree,
we can check the AJ-conjecture. Evaluating $P(-1,\sm,\sl)$ gives
\[P(-1,\sm,\sl) =  b^{-1}(-1,\sm) (\sl-1)(\sl-\sm^{-4pq})\sm^{r}(\sl+\sm^{-2r}),\]
which is equal to the $A$-polynomial of $C$ up to a nonzero factor in ${\mathbb Q}(\sm)$.

\subsection{$P$ is the Recurrence Polynomial of $C$.}

Next we show that the operator $P$ given above is a generator
of the  ideal $\wA_\sc$.
It amounts to show that if an operator $Q=D_2\sl^2+D_2\sl+D_0$,
where each $D_j\in \z[t^{\pm1}, \sm^{\pm1}]$, is an annihilator
of $J_{\sc, n}(t)$, then $Q=0$.

So suppose that $QJ_{\sc,n}(t)=0$, i.e.
\begin{equation}\label{degree at least three}
D_2J_{\sc,n+2}(t)+D_1J_{\sc,n+1}+D_0J_{\sc,n}(t)=0.
\end{equation}
Our goal is to show that $D_i=0$, $i=0,1,2$.

Using (\ref{s=2 first relation}) and (\ref{s=2 second relation}) we can transform  (\ref{degree at least three})
into
$$\label{degree at most 2}
\begin{array}{ll}0&=D_2(t^{-4rn-2r}J_{\st,2n+3}(t)-
t^{-4rn-6r}(\sm^{-r}J_{\st,2n+1}(t)-
t^{-4rn-2r}J_{\sc,n}(t)))\\&\quad+D_1(t^{-rn}J_{\st,2n+1}(t)-
t^{-4rn-2r}J_{\sc,n}(t))+D_0J_{\sc,n}(t)
\\&=(D_2t^{-8rn-8r}-D_1t^{-4rn-2r}+D_0)J_{\sc,n}(t)\\&
\quad+D_2t^{-2rn-2r}(t^{-8pqn-8pq}J_{\st,2n+1}(t)+t^{-4pqn-4pq} \delta_{2n+1})
\\&\quad+(-D_2t^{-6rn-6r}+D_1t^{-2rn})J_{\st,2n+1}(t)\\&
=(D_2t^{-8r(n+1)}-D_1t^{-4rn-2r}+D_0)J_{\sc,n}(t)\\
&\quad+
(D_2(t^{-2r(n+1)-8pq(n+1)}-t^{-6r(n+1)})
+D_1t^{-2rn})J_{\st, 2n+1}(t)\\
&\quad+
D_2t^{-2r(n+1)-4pq(n+1)}\delta_{2n+1}
\\&=D_2'J_{\sc,n}(t)+D_1'J_{\st,2n+1}(t)+D_0'.
\end{array}$$

If we can show that $D_i'=0$, $i=0,1,2$, then it will follow right away
that  $D_i=0$, $i=0,1,2$.
As  in Lemma \ref{degree argument for s odd}, we can show that $D_i'=0$, $i=0,1,2$, if $r$ is
not an integer between $0$ and $2pq$.

\end{document}